\documentclass[12pt, reqno]{amsart}
\usepackage{ amsmath,amsthm, amscd, amsfonts, amssymb, graphicx, color}
\usepackage[bookmarksnumbered, colorlinks, plainpages]{hyperref}
\newtheorem{thm}{Theorem}[section]
\newtheorem{cor}[thm]{Corollary}
\newtheorem{lem}[thm]{Lemma}
\newtheorem{prop}[thm]{Proposition}
\newtheorem{defn}[thm]{Definition}
\newtheorem{rem}[thm]{\bf{Remark}}

\newtheorem{exam}[thm]{Example}
\numberwithin{equation}{section}

\begin{document}

\title{Rings with $\Delta$-quasipolarity Property}

\author{Tugce Pekacar Calci}
\author{Serhat Emirhan Soycan}
\address{
Department of Mathematics, Ankara University, Ankara, Turkey}
\email{<tcalci@ankara.edu.tr>}
\address{
Ankara University Graduate School of National and Applied Sciences, Ankara, Turkey}
\email{<sesoycan@ankara.edu.tr>}


\begin{abstract}
This study provides a comprehensive investigation into the structure and properties of a novel class of rings known as $\Delta$-quasipolar rings, in which for every $a\in R$ there exisxt $p^2=p \in comm^2(a)$ such that $a+p \in \Delta(R)$. Here, $\Delta(R)$ is the largest Jacobson radical subring of $R$ which is closed with respect
to multiplication by units of $R$.
We have developed a general and comprehensive theoretical framework for $\Delta$-quasipolar rings. Within this framework, we have described the fundamental mathematical properties of these rings and examined their behavior under various classical algebraic structures, such as matrix rings and polynomial rings. Furthermore, we have investigated how this class of rings is related to other well-known ring types in mathematics, such as strongly $\Delta$-clean rings, uniquely clean rings, $J$-quasipolar rings and others.

 \vspace{2mm}

\end{abstract}

\maketitle

\section{Introduction}
Throughout this paper, all rings are assumed to be unital and associative. As usual, we denote the Jacobson radical, the set of nilpotent elements, the set of idempotent elements, and the group of units of a ring $R$ by $J(R)$, $\text{Nil}(R)$, $\text{Id}(R)$, and $U(R)$, respectively. The rings of $n \times n$ matrices and $n \times n$ upper triangular matrices over $R$ are denoted by $M_n(R)$ and $T_n(R)$, respectively. A ring is said to be \textit{abelian} if every idempotent element is central.

    Following \cite{Koliha}, we adopt the following notations. For an element $a\in R$, the \textit{commutant} and \textit{double commutant} of $a\in R$ are defined by $comm(a)=\{x \in R \mid ax=xa \}$ and $comm^2(a)=\{y \in R \mid xy=yx \text{ for all } x\in comm(a) \}$. Let $R^{qnil}=\{a\in R ~|~1+ax\in U(R) \text{ for all } x\in comm(a)\}$, and if $a\in R^{qnil}$, then $a$ is called \textit{quasinilpotent}. An element \( r \in R \) is said to be \emph{clean} if there exist an idempotent \( e \in R \) and a unit \( u \in R \) such that \( r = e + u \). Moreover, \( r \) is called \emph{strongly clean} if, in addition, the idempotent and the unit can be chosen to commute; that is, \( eu = ue \). Accordingly, a ring \( R \) is called \emph{clean} (respectively, \emph{strongly clean}) if every element of \( R \) is clean (respectively, strongly clean). An element $a\in R$ is called {\it quasipolar} if there exists $p^2=p\in R$ such that $p\in comm^2(a)$, $a+p\in U(R)$ and $ap\in R^{qnil}$. An idempotent satisfying the above conditions is called a spectral idempotent of $a$. Ying and Chen \cite{Ying} called a ring $R$ quasipolar if every element of $R$ is quasipolar. After these Cui and Chen investigated the quasipolarity in terms of Jacobson radical in \cite{Cui}. Let $a\in R$, $a$ is called \textit{$J$-quasipolar} if there exists $p^2=p\in comm^2(a)$ such that $a+p\in J(R)$. A ring $R$ is called $J$-quasipolar if every element of $R$ is $J$-quasipolar.

By analyzing the aforementioned definitions and noting that $\Delta(R)$ is a (possibly proper) subset of $J(R)$ which is not necessarily an ideal, a natural question arises: what is the structure of those rings $R$ in which the sum of the element $a$ of $R$ and the idempotent $p$ of $R$ is an element from $\Delta(R)$ where $p\in comm^2(a)$? The main goal of this paper is to investigate such rings and provide a detailed analysis of their structural properties.

Quasipolar-type rings have been extensively studied due to their rich internal structure and connections with various concepts in noncommutative ring theory. In particular, the notions of $J$-quasi-\linebreak polar rings and $\delta$-quasipolar rings have led to significant developments in understanding how ring elements can be decomposed into simpler algebraic parts. Motivated by this framework, the present work introduces and examines the class of $\Delta$-quasipolar rings. This terminology arises naturally from the preceding discussion. It is clear that all $J$-quasipolar rings are $\Delta$-quasipolar.

   In light of this, $\Delta(R)$ defined as $$
\Delta(R) = \{ x \in R \mid 1 - xu \in U(R) \text{ for all } u \in U(R) \},$$ has been the subject of extensive study as a natural generalization of the Jacobson radical (see \cite{Leroy} and\cite{Lam2}). Recently, a ring \( R \) is called \emph{strongly \( \Delta \)-clean} if every element of \( R \) can be written as the sum of an idempotent from \( R \) and an element from \( \Delta(R) \) that commute with each other, \cite{Danchev}.

    The aim of this article is to introduce $\Delta$-quasipolar rings. An element $a\in R$ is called \textit{$\Delta$-quasipolar} if there exists an idempotent $p\in comm^2(a)$ such that $a+p\in \Delta(R)$. Idempotents satisfying the above conditions called  $\Delta$-spectral idempotents. If every element of $R$ is $\Delta$-quasipolar, then $R$ is called $\Delta$-quasipolar ring. In section 2, we have described the fundamental mathematical properties of these rings and examined their behavior under various classical algebraic structures, such as Dorroh extension of a ring, polynomial rings and matrix rings. Furthermore, we have investigated how this class of rings is related to other well-known ring types in mathematics, such as strongly $\Delta$-clean rings, uniquely clean rings, $J$-quasipolar rings and others. We showed that every uniquely clean ring is $\Delta$-quasipolar. In addition, every $\Delta$-quasipolar ring is strongly $\Delta$-clean. We also proved that the matrix ring $M_n(R)$ is not $\Delta$-quasipolar in general. Lastly, we have determined the $\Delta(H_{(s,t)}(R))$ and showed that $H_{(s,t)}(R)$ is $\Delta$-quasipolar if and only if $R$ is $\Delta$-quasipolar. 

\section{$\Delta$-quasipolar rings}

Here, we begin with the structure that plays a key role in our study.
\[
\Delta(R) = \{r \in R \mid \forall u \in U(R),\ r + u \in U(R)\}\supseteq J(R).
\]

\begin{lem}\cite[Lemma 1]{Leroy}
 For any ring $R$, we have:
\begin{enumerate}
    \item $\Delta(R) = \{r \in R \mid \forall u \in U(R),\ ru + 1 \in U(R)\} = \{r \in R \mid \forall u \in U(R),\ ur + 1 \in U(R)\};$
    \item \textit{For any $r \in \Delta(R)$ and $u \in U(R)$, $ur, ru \in \Delta(R)$;}
    \item $\Delta(R)$ \textit{is a subring of $R$;}
    \item $\Delta(R)$ \textit{is an ideal of $R$ if and only if $\Delta(R) = J(R)$;}
    \item \textit{For any rings $R_i$, $i \in I$, $\Delta\left( \prod_{i \in I} R_i \right) = \prod_{i \in I} \Delta(R_i)$.}
\end{enumerate}
\end{lem}

\begin{lem}
    Let $R$ be a ring. If $U(R)\subseteq C(R)$, then $R^{qnil}\subseteq \Delta(R)$.
\end{lem}

\begin{proof}
    Let $R$ be a ring with $U(R)\subseteq C(R)$. Then for any $d\in \Delta(R)$, $1+du\in U(R)$. Since $U(R)\subseteq C(R)$, $d\in R^{qnil}$.
\end{proof}

We now introduce the key definition that underpins the framework of our investigation.
\begin{defn}
{\rm    An element $a\in R$ is called $\Delta$-quasipolar if there exists $p^2=p\in comm^2(a)$ such that $a+p\in \Delta (R)$; an idempotent $p$ is called a $\Delta$-spectral idempotent.  The ring $R$ is called $\Delta$-quasipolar if every element of $R$ is $\Delta$-quasipolar.}
\end{defn}

\begin{rem}
   {\rm $\Delta$-spectral idempotent need not be unique. Consider the ring $T_2(\Bbb Z_2)$ which is $\Delta$-quasipolar by Example \ref{counterex}. Here, $\left[%
\begin{array}{cc}
  1 & 1 \\
  0 & 0 \\
\end{array}%
\right] \in T_2(\Bbb Z_2)$ is $\Delta$-quasipolar with $\Delta$-spectral idempotents $\left[%
\begin{array}{cc}
  1 & 1 \\
  0 & 0 \\
\end{array}%
\right] $ and $\left[%
\begin{array}{cc}
  1 & 0 \\
  0 & 0 \\
\end{array}%
\right] $ since $\Delta(T_2(R)) = D_2(\Delta(R)) + J_2(R)$ where $J_2(R)$ the ideal of $T_2(R)$ consisting of all strictly upper triangular matrices and $D_2(R)$ the
subring of diagonal matrices by \cite{Leroy}.}
\end{rem}

At this point, we offer a few basic examples to demonstrate the routine calculations.
\begin{exam}\label{example}
   {\rm \begin{enumerate}
        \item Every element of $\Delta(R)$ is $\Delta$-quasipolar.
        \item Every $J$-quasipolar ring is $\Delta$-quasipolar. 
        \item If $a\in R$ is $\Delta$-quasipolar, then $-1-a$ is $\Delta$-quasipolar.
    \end{enumerate}}
\end{exam}

The following theorem provides a link between 
$\Delta$-quasipolar rings and 
strongly $\Delta$-clean rings.

\begin{thm}\label{stdelta}
    Every $\Delta$-quasipolar ring is strongly $\Delta$-clean. Converse holds when $R$ is abelian.
\end{thm}

\begin{proof}
    Let $R$ be a $\Delta$-quasipolar ring and $a\in R$. Then $a-1\in R$ and $a-1$ is $\Delta$-quasipolar. Then there exists $p^2=p\in comm^2(a-1)$ such that $a-1+p\in \Delta (R)$.  Hence there exists $d\in \Delta(R)$ such that $a-1+p=d$. Therefore $a=1-p+d$. So $a$ is strongly $\Delta$-clean. Now let $R$ be an abelian strongly $\Delta$-clean ring and $a\in R$. Then $a+1$ is strongly $\Delta$-clean. Hence there exists $p^2=p\in comm(a)$ such that $a+1=p+d$. This implies $a+1-p=d$. Since $R$ is abelian, $1-p\in comm^2(a)$. This completes the proof.
\end{proof}

Exploiting Theorem \ref{stdelta} and \cite[Corollary 4.3]{Danchev}, we have following:

\begin{cor}\label{TnR}
    Let $R$ be an abelian ring. Then the following statements are equivalent.
    \begin{enumerate}
        \item $R$ is $\Delta$-quasipolar.
        \item $R$ is strongly $\Delta$-clean.
        \item R is uniquely clean.
    
    \end{enumerate}
\end{cor}

A natural question arises as to whether 
$\Delta$-quasipolar rings are abelian. The next example sheds light on this and further reveals that 
$\Delta$-quasipolar rings are not necessarily uniquely clean.

\begin{exam}\label{counterex}
   {\rm It is obvious that $T_2(\Bbb{Z}_2)$ is $J$-quasipolar (see \cite{Cui}), so is $\Delta$-quasipolar. However it is not abelian, thus it is not uniquely clean by \cite[Theorem 2.1]{Chen}.}
\end{exam}
In addition, leveraging Theorem \ref{stdelta} in conjunction with \cite[Theorem 3.4]{Danchev}, we arrive at the following results.

\begin{cor}
    Let $R$ be a ring. 

    \begin{enumerate}
        \item If $R$ is uniquely $\Delta$-clean, then $R$ is $\Delta$-quasipolar.
        \item If $R$ is uniquely clean, then $R$ is $\Delta$-quasipolar.
    \end{enumerate}
\end{cor}

%\begin{defn}
 %   Let $R$ be a ring. An element $a\in R$ is called $\delta$-quasipolar if there exists $p^2=p\in comm^2(a)$ such that $a+p\in \delta(R)$ and $p$ is called a $\delta$-spectral idempotent of $a$. The ring $R$ is said to be $\delta$-quasipolar if every element of $R$ is $\delta$-quasipolar.
%\end{defn}

%$\Bbb{Q}$ is a $\delta$-quasipolar ring but it is obviously not $\Delta$-quasipolar.

\begin{lem}\label{2}
    Let $R$ be a ring. If $R$ is $\Delta$-quasipolar, then $2\in \Delta(R)$.
\end{lem}

\begin{proof}
    Let $R$ be a $\Delta$-quasipolar ring and $1\in R$. Then there exists an idempotent $p\in comm^2(1)$ such that $1+p\in \Delta(R)$. Since $1-2p\in U(R)$, $1+p+1-2p=2-p\in \Delta(R)$. So, $(2-p)(1+p)=2\in \Delta(R)$.
\end{proof}

Let $R$ be a ring. One may wonder whether $R^{qnil}$ is equal to $\Delta(R)$ or not. Following example shows that this is not the case.

\begin{exam}
  {\rm Consider the ring $M_n(\Bbb{Q})$. It is known that $\Delta(M_n(\Bbb{Q})) = J(M_n(\Bbb{Q})) = M_n(J(\Bbb{Q}))$ by \cite{Leroy} and $(M_n(\Bbb{Q}))^{qnil}=Nil(M_n(\Bbb{Q}))$. Clearly, the matrix $\left[%
\begin{array}{cc}
  0 & 1 \\
  0 & 0 \\
\end{array}%
\right] \in (M_n(\Bbb{Q}))^{qnil}$. But not in $\Delta(M_n(\Bbb{Q}))$. So, we presented the following lemma:}
\end{exam}

\begin{lem}
    Let $R$ be a ring. If $a$ is $\Delta$-quasipolar, then $u^{-1}au$ is $\Delta$-quasipolar.
\end{lem}

\begin{proof}
    Let $a\in R$ be $\Delta$-quasipolar. Then there exists $p^2=p\in comm^2(a)$ such that $a+p\in \Delta(R)$. Let $v=u^{-1}$. Suppose that $x\in comm(vau)$. Then we have $vaux=xvau$. Hence, $(uxv)a=a(uxv)$. So $uxv\in comm(a)$. Since $p\in comm^2(a)$, $p(uxv)=(uxv)p$. Therefore, $(vpu)x=x(vpu)$. So $(vpu)^2=vpu \in comm^2(vau)$. Now, since $a+p\in \Delta(R)$ and $\Delta(R)$ closed under multiplication with unit, $v(a+p)u=vau+vpu\in \Delta(R)$, as asserted.
\end{proof}

\begin{lem}\label{invidem}
    Let $R$ be a $\Delta$-quasipolar ring and $u\in U(R)$. Then the $\Delta$-spectral idempotent of $u$ is $1$.
\end{lem}

\begin{proof}
    Let $R$ be a $\Delta$-quasipolar ring and $u\in U(R)$. Then there exists $p^2=p\in comm^2(u)$ such that $u+p\in \Delta(R)$. Hence, $1+u^{-1}(u+p)=1+1+u^{-1}p=2+u^{-1}p\in U(R)$. By Lemma \ref{2}, $2\in \Delta(R)$. Therefore $-2-u^{-1}p+2=-u^{-1}p\in U(R)$. This implies that $p\in U(R)$. So $p=1$.
\end{proof}

\begin{lem}
    Let $R$ be a $\Delta$-quasipolar ring and $a\in Nil(R)$. Then  $\Delta$-spectral idempotent of $a$ is $0$.
\end{lem}

\begin{proof}
    Let $R$ be a $\Delta$-quasipolar ring and $a^n=0$ for $n\in \Bbb{Z}$. Then there exists $p^2=p\in comm^2(a)$ such that $(a+p)\in \Delta(R)$. Since $\Delta(R)$ is a subring, $(a+p)^n\in \Delta(R)$. Hence $(a+p)=a^n+\left(%
\begin{array}{c}
  n \\
  1 \\
\end{array}%
\right)a^{n-1}p+\left(%
\begin{array}{c}
  n \\
  2 \\
\end{array}%
\right)a^{n-2}p+\cdots +\left(%
\begin{array}{c}
  n \\
  n-2 \\
\end{array}%
\right)a^2p+\left(%
\begin{array}{c}
  n \\
  n-1 \\
\end{array}%
\right)ap+p\in \Delta(R)$. Since $\Delta(R)$ is closed with respect to multiplication by nilpotent elements, $a^{n-1}(a+p)^n=a^{n-1}p\in \Delta(R)$. Therefore $(a+p)^{n}-na^{n-1}p=\left(%
\begin{array}{c}
  n \\
  2 \\
\end{array}%
\right)a^{n-2}p+\cdots +\left(%
\begin{array}{c}
  n \\
  n-2 \\
\end{array}%
\right)a^2p+\left(%
\begin{array}{c}
  n \\
  n-1 \\
\end{array}%
\right)ap+p\in \Delta(R)$. Multiplying the last equation by $a^{n-2}$, we have $a^{n-2}p\in \Delta(R)$. So $(a+p)^{n}-na^{n-1}p-\left(%
\begin{array}{c}
  n \\
  2 \\
\end{array}%
\right)a^{n-2}p=\left(%
\begin{array}{c}
  n \\
  3 \\
\end{array}%
\right)a^{n-3}p+\cdots +p\in \Delta(R)$. If we continue similarly, we have $p\in \Delta(R)$. Then $1-p\in U(R)$ and this implies $p=0$.
\end{proof}

Now we give a direct consequence of above lemma.

\begin{cor}
    Let $R$ be a ring. If $R$ is $\Delta$-quasipolar, then $Nil(R)\subseteq \Delta(R)$.
\end{cor}

\begin{lem}
    Let $R$ be a $\Delta$-quasipolar ring. For every $a\in R$, there exist $p^2=p\in comm^2(a)$ such that $a+p\in U(R)$.
\end{lem}

\begin{proof}
    Let $R$ be a $\Delta$-quasipolar ring and $a\in R$. Then $a-1$ is $\Delta$-quasipolar, so there exists $p^2=p\in comm^2(a-1)$ such that $a-1+p\in \Delta(R)$. Hence, $1+a-1+p=a+p\in U(R)$ where $p\in comm^2(a)$.
\end{proof}

\begin{lem}\label{deltaideal}
    Let $R$ be a $\Delta$-quasipolar ring and $2\in U(R)$. Then $\Delta(R)$ is an ideal.
\end{lem}

\begin{proof}
    Let $x\in R$ and $y\in \Delta(R)$. Since $R$ is $\Delta$-quasipolar, there exist $p^2=p\in comm^2(x)$, $d\in \Delta(R)$ such that $x+p=d$. Multiplying the equation by $y$, we get $xy+py=dy$. By \cite[Corollary 2]{Leroy}, $py\in \Delta(R)$. So $xy\in \Delta(R)$. Similarly, $yx\in \Delta(R)$.
\end{proof}

\begin{prop}
    The following are equivalent for a ring $R$ :
    \begin{enumerate}
        \item $R$ is local and $\Delta$-quasipolar.
        \item $R$ is $\Delta$-quasipolar with only idempotents $0$ and $1$.
        \item $R/J(R)\cong \Bbb{Z}_2$ 
    \end{enumerate}
\end{prop}

\begin{proof}
   (1) $\Rightarrow$ (2) It is clear.
   
   (2) $\Rightarrow$ (3) Let $R$ be a $\Delta$-quasipolar ring with only idempotents $0$ and $1$. Then $R$ is strongly $\Delta$-clean with only idempotents $0$ and $1$ by Theorem \ref{stdelta}. Then by using \cite[Lemma 2.4]{Danchev} and \cite[Lemma 14]{Nicholson}, $R$ is local. Let $u\notin J(R)$. We will prove that $1-u\in J(R)$. Assume that $1-u\notin J(R)$. Since $R$ is local, $1-u\in U(R)$. By Lemma \ref{invidem}, $\Delta$-spectral idempotent of $1-u$ is $1$. Then $1+1-u=2-u\in \Delta(R)$ and this implies $2\in U(R)$. So by Lemma \ref{deltaideal}, $\Delta(R)$ is an ideal. Moreover, $2\in \Delta(R)$ by Lemma \ref{2}, so this is a contradiction. Therefore, $1-u\in J(R)$ as desired.

   (3) $\Rightarrow$ (1) It is clear by \cite[Proposition 2.11]{Cui} and Example \ref{example}(2).\end{proof}

   \begin{lem}\label{ann}
    Let $R$ be a ring and $a+p=d$ is a $\Delta$-quasipolar representation of $a$ in $R$. Then $ann_l(a)\subseteq ann_l(p)$ and $ann_r(a)\subseteq ann_r(p)$.
\end{lem}

\begin{proof}
    Let $a+p=d$ be a $\Delta$-quasipolar decomposition of $a$ in $R$. Then $a^2=d^2-2dp+p$. By \cite[Lemma 2.6]{Danchev}, there exists $d'\in \Delta(R)$ such that $a^2=d'+p$. Let $xa=0$. Then $xa^2=xd'+xp=0$. Multiplying by $p$, we have $0=xd'p+xp=xpd'+xp=xp(1+d')$. Since $1+d'\in U(R)$, $xp=0$ as desired. Similarly, it can be shown that $ann_r(a)\subseteq ann_r(p)$.
\end{proof}

\begin{lem}
    Let $R$ be a ring. If $R$ is abelian $J$-clean, then $R$ is $\Delta$-quasipolar.
\end{lem}

\begin{proof}
 Let $a\in R$. Since $R$ is $J$-clean, there exist $e^2=e$ and $j\in J(R)$ such that $-a=e+j$. This implies that $a+e\in J(R) \subseteq \Delta(R)$. Since $R$ is abelian, $e^2=e\in comm^2(a)$. So $R$ is $\Delta$-quasipolar, as desired.
\end{proof}

\begin{prop}
    Let $R$ be a $\Delta$-quasipolar ring and $\Delta(R)=J(R)$. Then $R$ is strongly $\pi$-regular if and only if $J(R)=R^{qnil}=Nil(R)=\Delta(R)$.
\end{prop}

\begin{proof}
    
$(\Rightarrow:)$ Let $a \in R^{qnil}$. Then for any $x \in comm(a)$, $1 - ax$ is invertible. By hypothesis, there exist a positive integer $m$ and $b \in R$ such that $a^m = a^{m+1}b$. Since $b \in comm(a)$ (\cite{Lam}, Page 347, Exercise 23.6(1)), $a^m = 0$. Hence $a \in Nil(R)$ and so $R^{qnil} \subseteq Nil(R)$. To prove $Nil(R) \subseteq \Delta(R)$, let $a \in Nil(R)$. By hypothesis, there exists $p^2 = p \in comm^2(1 - a)$ such that $1 - a + p \in \Delta(R)$. Since $1 - a$ is invertible, $p = 1$ by Lemma \ref{invidem}. Hence, $2 - a \in \Delta(R)$. Also, $2 \in \Delta(R)$ by Lemma \ref{2}, we then have $a \in \Delta(R)$.

$(\Leftarrow:)$ Let $a \in R$. There exists $p^2 = p \in comm^2(-1 + a)$ such that $-1 + a + p \in \Delta(R)$. Set $u = -1 + a + p $, by hypothesis, $u\in Nil(R)$. Then $u+1=a + p$ is invertible and $ap = up$ is nilpotent so that $a^n p = 0$ for some positive integer $n$. So $a^n = a^n(1 - p) = (u + (1 - p))^n(1 - p) = (u + 1)^n(1 - p) = (a + p)^n(1 - p) = (1 - p)(a + p)^n$. By \cite[Proposition 1]{Nicholson}, $a$ is strongly $\pi$-regular. This completes the proof. 
\end{proof}

\begin{lem}
    Let $I$ be an index set and for every $i\in I$, $R_i$ be rings. Then $R_i$ is $\Delta$-quasipolar if and only if $\prod_{i\in I}R_i$ is $\Delta$-quasipolar.
\end{lem}

\begin{proof}
    We proceed with the proof for $I=\{1,2\}$. It can be easily generalized $I=\{1,\dots, n\}$. Let $R_1$ and $R_2$ be $\Delta$-quasipolar rings and $(r_1,r_2) \in R_1\times R_2$. Since $R_1$ and $R_2$ are $\Delta$-quasipolar, there exist $p_1^2=p_1\in comm^2(r_1)$, $p_2^2=p_2\in comm^2(r_2)$ such that $p_1+r_1\in \Delta(R_1)$, $p_2+r_2\in \Delta(R_2)$. Therefore $((p_1,p_2))^2=(p_1,p_2)\in comm^2((r_1,r_2))$. By \cite[Lemma 1]{Leroy}, $\Delta(R_1\times R_2)=\Delta(R_1)\times \Delta(R_2)$, so $(r_1,r_2)+(p_1,p_2)=(r_1+p_1,r_2+p_2)\in \Delta(R_1\times R_2)$. Conversely, let $R_1\times R_2$ be $\Delta$-quasipolar ring. Let $a\in R_1$. Then $(a,0)\in R_1\times R_2$ is $\Delta$-quasipolar. Hence there exist $(p_1,p_2)^2=(p_1,p_2)\in comm^2((a,0))$ such that $(a,0)+(p_1,p_2) \in \Delta(R_1\times R_2)= \Delta(R_1)\times \Delta(R_2)$. Therefore $a+p_1\in \Delta(R_1)$. Also it is obvious that $p_1^2=p_1$. Now let $x\in comm(a)$. Then $(x,0)(a,0)=(xa,0)=(ax,0)=(a,0)(x,0)$. Since $(p_1,p_2)\in comm^2((a,0))$, $(p_1,p_2)(x,0)=(p_1x,0)=(xp_1,0)=(x,0)(p_1,p_2)$. So $p_1\in comm^2(a)$,  as desired. $\Delta$-quasipolarity of $R_2$ can be shown similarly.
\end{proof}

%Following two corollaries are direct results of \cite{Leroy}, Theorem 3.
%\begin{cor}
    %Let $R$ be a $\Delta$-quasipolar ring and $T$ be the subring of $R$ generated by $U(R)$. Then $T$ is $J$-quasipolar.
%\end{cor}

%\begin{cor}
 %   Let $R$ be a $\Delta$-quasipolar ring and $T$ be the subring of $R$ generated by $U(R)$. Then for any subring of $R$ than contains $T$ is $\Delta$-quasipolar.
%\end{cor}

\begin{thm}\label{eRe}
    Let $R$ be a $\Delta$-quasipolar ring and $e^2=e\in R$. Then $eRe$ is $\Delta$-quasipolar.
\end{thm}

\begin{proof}
    Let $a\in eRe$. Since $R$ is $\Delta$-quasipolar, there exists $p^2=p\in comm^2(a)$ such that $a+p\in \Delta(R)$. Since $a\in eRe$, we have $a=ea=ae=eae$. It is clear that $(1-e)\in ann_r(a)\cap ann_l(a)$. Then we have $(1-e)\in ann_r(p)\cap ann_l(p)$ by Lemma \ref{ann}. Hence $p=pe=ep=epe$ and this implies $(epe)^2=epe$. Thus  we have $a+epe=a+p \in \Delta(R)\cap eRe$. By \cite[Lemma 2.19]{Danchev}, $a+e\in \Delta(eRe)$. Now, let $b\in eRe$ such that $ab=ba$. Since $p\in comm^2(a)$, $pb=bp$. So we have $epeb=bepe$ as desired.
\end{proof}

Let $R$ and $V$ be rings and $V$ be an $(R,R)$-bimodule that is also a ring with
$$
(vw)r = v(wr), \quad (vr)w = v(rw), \quad \text{and} \quad (rv)w = r(vw)
$$
for all $v, w \in V$ and $r \in R$. The \textit{Dorroh extension} $D(R,V)$ of $R$ by $V$ is defined as the ring consisting of the additive abelian group $R \oplus V$ with multiplication
$$
(r,v)(s,w) = (rs, rw + vs + vw)
$$
where $r, s \in R$ and $v, w \in V$.

\begin{thm}
    Let $R$ be a ring. Then we have the following.
    \begin{enumerate}
        \item If $D(R,V)$ is $\Delta$-quasipolar, then $R$ is $\Delta$-quasipolar.
        \item If the following conditions are satisfied, then $D(R,V)$ is $\Delta$-quasipolar.
        \begin{enumerate}
            \item[(i)] $R$ is $\Delta$-quasipolar.
            \item[(ii)] If $e^2=e$, then $ev=ve$ for all $v\in V$.
            \item[(iii)] If $v\in V$, then there exists $w\in V$ such that \\$v+w+vw=0$. 
        \end{enumerate}
    \end{enumerate}
\end{thm}

\begin{proof}
    (1) Let $r\in R$, then $(r,0)\in D(R,V)$ is $\Delta$-quasipolar. So there exists $e^2=e\in comm^2((r,0))$ such that $(r,0)+e\in \Delta(D(R,V))$. Let $e=(p,q)$ for some $p\in R$ and $q \in V$. Then $p^2=p\in R$. Now let $xr=rx$ for some $x\in R$. Then we have $(x,0)(r,0)=(xr,0)=(rx,0)=(r,0)(x,0)$. Since $(p,q)\in comm^2((r,0))$, we have $(p,q)(x,0)=(px,qx)=(xp,xq)=(x,0)(p,q)$. So $px=xp$ and this implies $p\in comm^2(r)$. Also we have $(r,0)+(p,q)=(r+p,q)\in \Delta(D(R,V))$. Therefore $r+p\in \Delta(R)$. This completes the proof.

    (2) Assume that (i), (ii), (iii) hold. Let $(r,v)\in D(R,V)$.

    \textbf{Claim:} $D(\Delta(R),V)\subseteq \Delta(D(R,V))$. Let $(x,y)\in D(\Delta(R),V))$ and $(u,v)\in U(D(R,V))$. Then $1+(x,y)(u,v)=(1+xu,xv+yu+yv)\in U(D(R,V))$ by \cite[Proposition 7]{Zhou} as desired.

    Now there exists $p^2=p\in comm^2(r)$ such that $r+p\in \Delta(R)$. Since $p^2=p$, $(p,0)^2=(p,0)$ and by claim $(r,v)+(p,0)=(r+p,v)\in \Delta(D(R,V))$. Finally let $(x,y)\in D(R,V)$ such that $(x,y)(r,v)=(r,v)(x,y)$. Since $p\in comm^2(r)$, we have $(p,0)(x,y)=(x,y)(p,0)$ as desired.
\end{proof}

\begin{prop}
    For any ring $R$, the following hold.
    \begin{enumerate}
        \item $R[x]$ is not $\Delta$-quasipolar.
        \item For every $n\geq 2$, the matrix ring $M_n(R)$ is not $\Delta$-quasipolar. 
    \end{enumerate}
\end{prop}

\begin{proof}
    With the aid of Corollary \ref{stdelta}, if $M_n(R)$ or $R[x]$ were $\Delta$-quasipolar, then they would be strongly $\Delta$-clean. However the authors showed otherwise \cite[Proposition 2.21]{Danchev}.
\end{proof}

{\bf The rings $H_{(s,t)}(R)$:} Let $R$ be a ring and  $s, t\in
C(R)\cap U(R)$. Set {\footnotesize$$H_{(s,t)}(R) = \left
\{\begin{bmatrix}a&0&0\\c&d&e\\0&0&f
\end{bmatrix}\in M_3(R)\mid a, c, d, e, f\in R, a - d = sc, d - f = te\right \}.$$}
Then $H_{(s,t)}(R)$ is a subring of $M_3(R)$.\\
Note that $U(H_{(s,t)}(R))=\left
\{\begin{bmatrix}a&0&0\\c&d&e\\0&0&f
\end{bmatrix}\in M_3(R)\mid a, d, f\in U(R)\right \}$. Now using that we have a useful lemma.

\begin{lem}\label{Hst}
    Let $R$ be a ring, then $$\Delta (H_{(s,t)}(R)) =\left
\{\begin{bmatrix}a&0&0\\c&d&e\\0&0&f
\end{bmatrix}\in M_3(R)\mid a, d, f\in \Delta(R)\right \}. $$
\end{lem}

\begin{proof}
    Let $\begin{bmatrix}a&0&0\\c&d&e\\0&0&f
\end{bmatrix}\in \Delta (H_{(s,t)}(R))$. Then for any $\begin{bmatrix}x&0&0\\y&z&p\\0&0&q
\end{bmatrix} \in U(H_{(s,t)}(R))$, we have $$\begin{bmatrix}a&0&0\\c&d&e\\0&0&f
\end{bmatrix}+\begin{bmatrix}x&0&0\\y&z&p\\0&0&q
\end{bmatrix}=\begin{bmatrix}a+x&0&0\\c+y&d+z&e+p\\0&0&f+q
\end{bmatrix}\in U(H_{(s,t)}(R)).$$ So for any $x,z,q\in U(R)$, we have $a+x,d+z, f+q\in U(R)$ and this implies $a,d,f\in \Delta(R)$. Now let $\begin{bmatrix}a&0&0\\c&d&e\\0&0&f
\end{bmatrix}\in H_{(s,t)}(R)$ and $a,d,f\in \Delta(R)$. Let $\begin{bmatrix}x&0&0\\y&z&p\\0&0&q
\end{bmatrix} \in U(H_{(s,t)}(R))$. Then $x,z,q\in U(R)$. Since $a,d,f\in \Delta(R)$ and $x,z,q\in U(R)$, we have $a+x,d+z,f+q\in U(R)$. Therefore we have $$\begin{bmatrix}a&0&0\\x&d&e\\0&0&f
\end{bmatrix}+\begin{bmatrix}x&0&0\\y&z&p\\0&0&q
\end{bmatrix}=\begin{bmatrix}a+x&0&0\\c+y&d+z&e+p\\0&0&f+q
\end{bmatrix} \in U(H_{(s,t)}(R)),$$ as asserted.
\end{proof}

\begin{thm}
    Let $R$ be a ring and $S:=H_{(s,t)}(R)$. Then $S$ is $\Delta$-quasipolar if and only if $R$ is $\Delta$-quasipolar.
\end{thm}

\begin{proof}
    Let $S$ be a $\Delta$-quasipolar ring. Set $E=\begin{bmatrix}1&0&0\\s^{-1}&0&0\\0&0&0
\end{bmatrix}$. Then $ESE\cong R$. By Theorem \ref{eRe}, $R$ is $\Delta$-quasipolar. Conversely, let $R$ be a $\Delta$-quasipolar ring. Let $\begin{bmatrix}a&0&0\\c&d&e\\0&0&f
\end{bmatrix}\in S$. Since $R$ is $\Delta$-quasipolar, there exist some $\alpha,\beta, \theta$ idempotents such that $a+\alpha=\delta_1$, $d+\beta=\delta_2$, $f+\theta=\delta_3$, $\delta_1,\delta_2,\delta_3\in \Delta(R)$ and $\alpha\in comm^2(a),\beta\in comm^2(d),\theta \in comm^2(f)$. Set $\gamma=s^{-1}(\alpha-\beta)$ and $\phi=t^{-1}(\beta-\theta)$, then $\begin{bmatrix}\alpha&0&0\\\gamma&\beta&\phi\\0&0&\theta
\end{bmatrix} \in S$ is an idempotent. Also $\begin{bmatrix}a&0&0\\c&d&e\\0&0&f
\end{bmatrix}+\begin{bmatrix}\alpha&0&0\\\gamma&\beta&\phi\\0&0&\theta
\end{bmatrix}=\begin{bmatrix}\delta_1&0&0\\c+\gamma&\delta_2&e+\phi\\0&0&\delta_3
\end{bmatrix}\in \Delta(S)$ by Lemma \ref{Hst}. Now let $$\begin{bmatrix}x&0&0\\y&z&p\\0&0&q
\end{bmatrix}\begin{bmatrix}a&0&0\\c&d&e\\0&0&f
\end{bmatrix}=\begin{bmatrix}a&0&0\\c&d&e\\0&0&f
\end{bmatrix}\begin{bmatrix}x&0&0\\y&z&p\\0&0&q
\end{bmatrix}.$$ This implies $ax=xa,dz=zd$,$fq=qf$,$cx+dy=ya+zc$. Then $$\begin{bmatrix}\alpha&0&0\\\gamma&\beta&\phi\\0&0&\theta
\end{bmatrix}\begin{bmatrix}x&0&0\\y&z&p\\0&0&q
\end{bmatrix}=\begin{bmatrix}x&0&0\\y&z&p\\0&0&q
\end{bmatrix}\begin{bmatrix}\alpha&0&0\\\gamma&\beta&\phi\\0&0&\theta
\end{bmatrix}$$\\ since $\alpha\in comm^2(a),\beta\in comm^2(d),\theta \in comm^2(f)$. This completes the proof.
\end{proof}

\end{document}